\documentclass[12pt]{amsart}
\usepackage{amsmath}
\usepackage{amssymb}
\usepackage{enumerate}

\makeatletter
\@namedef{subjclassname@2020}{%
  \textup{2020} Mathematics Subject Classification}
\makeatother

\usepackage[T1]{fontenc}
\newtheorem{theorem}{Theorem}[section]
\newtheorem{corollary}[theorem]{Corollary}
\newtheorem{lemma}[theorem]{Lemma}
\newtheorem{proposition}[theorem]{Proposition}
\newtheorem{problem}[theorem]{Problem}

\theoremstyle{definition}

\newtheorem{remark}[theorem]{Remark}

\newtheorem{notation}[theorem]{Notation}
\newtheorem{construction}[theorem]{Construction}

\numberwithin{equation}{section}

\newcommand{\btu}{\bigtriangleup}

\newcommand{\fF}{\mathfrak F}

\newcommand{\cA}{\mathcal A}

\newcommand{\cC}{\mathcal C}

\newcommand{\fA}{\mathfrak A}
\newcommand{\fB}{\mathfrak B}

\newcommand{\cG}{\mathcal G}

\newcommand{\ult}{\protect{ult}}
\newcommand{\clop}{\protect{\rm Clop}}
\newcommand{\la}{\langle}
\newcommand{\ra}{\rangle}
\newcommand{\sub}{\subseteq}
\newcommand{\eps}{\varepsilon}
\newcommand{\er}{\mathbb R}
\newcommand{\sm}{\setminus}
\newcommand{\con}{\mathfrak c}

\newcommand{\vf}{\varphi}

\newcommand{\wh}{\widehat}

\newcommand{\oo}{\protect{\omega_1}}
\newcommand{\azm}{\aleph_0{\rm -monolithic}}
\newcommand{\cyl}{\protect{\rm cyl}}
\newcommand{\lo}{{\rm Lim}(\omega_1)}

\begin{document}

%%%%% To ease editing, for IMPAN journals add:

\baselineskip=15pt

%%%%%%%%%%%%%%%%

\date{}

\title{Monolithic spaces of measures}

\author[G.\ Plebanek]{Grzegorz Plebanek}
\address{Instytut Matematyczny\\ Uniwersytet Wroc\l awski\\ Pl.\ Grunwaldzki 2/4\\
50-384 Wroc\-\l aw\\ Poland} \email{grzes@math.uni.wroc.pl}

\subjclass[2020]{Primary 28A33, 46E27, 03E50}
\keywords{$\aleph_0$-monolithic, space of measures, Martin's axiom, Jensen's principle}

\begin{abstract}
For a compact space $K$ we consider the space $P(K)$, of probability regular Borel measures on $K$, equipped with the $weak^\ast$ topology inherited from
$C(K)^\ast$. We discuss possible characterizations of those compact spaces $K$ for which $P(K)$ is $\azm$.
The main result states that under $\diamondsuit$ there exists a nonseparable Corson compact space $K$ such that  $P(K)$ is $\azm$ but
$K$ supports  a measure of uncountable type.
\end{abstract}

\maketitle

\section{Introduction}
A compact space $K$ is $\azm$ if every separable subspace of $K$ is metrizable.
Typical example of such spaces are those  closely related to functional analysis: Eberlein compacta  and, more generally, Corson compacta, i.e.\
spaces that can be embedded into
\[\Sigma(\er^\kappa)=\{x\in\er^\kappa: |\{\alpha: x_\alpha\neq 0\}|\le\omega\},\]
for some $\kappa$.
In fact, there is a monotone version of monolithicity that implies Corson compactness, see Gruenhage \cite{Gr12}.

Given a compact space $K$, we denote by $P(K)$ the space of probability regular Borel measures on $K$ and we always equip  $P(K)$ with the $weak^\ast$ topology
inherited from $C(K)^\ast$, the dual space of the space $C(K)$ of continuous functions. This means that the topology on $P(K)$ is determined by continuity of all the mappings \[ P(K)\ni \mu\to \int_K g\;{\rm d}\mu, \quad g\in C(K).\]

We investigate here  for which compact spaces $K$, the space $P(K)$ is $\azm$.
It is easy to check that $P(K)$ is $\azm$ if and only if $B_{C(K)^\ast}$, the dual unit ball, is $\azm$ in its $weak^\ast$ topology.
Monolithicity of dual unit balls of Banach spaces emerged quite naturally in a number of papers devoted to investigating some isomorphic properties of
Banach spaces related to the space $c_0$,
see Kalenda and Kubi\'s \cite{KK12}, Ferrer, Koszmider and Kubi\'s \cite{FKK13}, Correa and Tausk \cite{CT13}, Ferrer \cite{Fe15}, Correa \cite{Co19}.

Recall that a measure $\mu\in P(K)$ is of type $\kappa$ if $\kappa$ is the density of the space of integrable functions $L_1(\mu)$.
Equivalently, the (Maharam) type of $\mu\in P(K)$ can be defined as the
minimal cardinality of a family $\cC$ of Borel subsets of $K$ having the property that
\[ \inf\{\mu(B\bigtriangleup C): C\in\cC\}=0 \mbox{  for every } B\in Bor(K).\]

In the next section we collect a number of essentially known results;
put together they will  explain  that monolithicity of spaces of measures is easy to handle under Martin's axiom
MA$(\oo)$. Then, at the end of section 2 we formulate a problem, if $\aleph_0$-monolithicity of $P(K)$ can be  characterized by
the property that every $\mu\in P(K)$ is supported by a metrizable subspace of $K$.   Our main objective is to demonstrate  that this is not the case;
assuming the Diamond Principle, we construct  a counterexample in section 4 and analyze the resulting space in section 5.
Section 3 contains some auxiliary results on measures on Boolean algebras.

\section{ Monolithicity under Martin's axiom}

Given a compact space $K$ and $\mu\in P(K)$, by the support of $\mu$ we mean the uniquely determined smallest closed subset of $K$ of measure one. We recall first the following well-known fact.

\begin{lemma}\label{added}
If a compact space $K$ is the support of a measure $\mu\in P(K)$ of type $\omega$ then the space $P(K)$ is separable.
\end{lemma}

The question for which compacta $K$ the space $P(K)$ is separable was investigated by Talagrand \cite{Ta80a} and
M\"agerl and Namioka \cite{MN80}.  Lemma \ref{added} follows from the fact that if a measure $\mu\in P(K)$ of countable type is positive on every nonempty opens subset of $K$ then one can easily define a sequence of $\mu_n\in P(K)$ such that for every nonempty open $U\sub K$ there is $n$
with $\mu_n(U)>1/2$.   In turn, this  gives rise to a positive isomorphic embedding $C(K)\to \ell_\infty$
and then the dual operator maps the $weak^\ast$ separable space $P(\beta\omega)$  onto $P(K)$.

It is perhaps worth recalling (though not needed later)
that Talagrand \cite{Ta80a} constructed  under CH two examples showing that the implication of Lemma \ref{added} cannot be reversed and that separability of $P(K)$ does not follow from separability of $C(K)^\ast$. Such examples were later constructed in the usual set theory,
see \cite{DP08} and \cite{APR14}.

The observations given in this section build on the following two results.

\begin{theorem}[Arkhangel'ski\u{\i} and \ Shapirovski\u{\i} \cite{AS87}] \label{thm1}
Under MA$(\oo)$, every compact $\aleph_0$-monolithic $ccc$ space is metrizable.
\end{theorem}

\begin{theorem}[Fremlin \cite{Fr97}] \label{thm2}
Under MA$(\oo)$, if a compact space $K$ carries a measure of uncountable type then $K$ can be continuously mapped onto
$[0,1]^{\oo}$.
\end{theorem}

We start by noting  basic facts.

\begin{lemma}\label{lem}
Let $K$ be a compact space.

\begin{enumerate}[(a)]
\item If $P(K)$ is $\aleph_0$-monolithic  then so is $K$.
\item If the support of every measure $\mu\in P(K)$ is metrizable then the space $P(K)$ is $\azm$.
\end{enumerate}
\end{lemma}

\begin{proof}
Clause $(a)$ follows from the fact that $K$ embeds into $P(K)$ via the mapping
\[  K\ni x \to \delta_x\in P(K),\]
where $\delta_x$ is the Dirac measure at $x$.

To check (b) take any sequence of $\mu_n\in P(K)$ and consider the measure
\[ \nu=\sum_{n=1}^\infty 2^{-n}\mu_n\in P(K).\]
Then the support $S$ of $\nu$ is metrizable
and hence $P(S)$ is metrizable too; moreover,  $\overline{\{\mu_n: n=1,2,\ldots\}}\sub P(S)$, and we are done.
\end{proof}

\begin{proposition} \label{prop}
Suppose that $K$ is a compact space such that  $P(K)$ is $\aleph_0$-monolithic.

\begin{enumerate}[(a)]
\item Then every $\mu\in P(K)$ is of type $\le \omega_1$.
\item If $\mu\in P(K)$ is of type $\omega$ then the support of $\mu$ is metrizable.
\item Under MA$(\omega_1)$, the support of every $\mu\in P(K)$ is metrizable.
\end{enumerate}
\end{proposition}

\begin{proof}
Talagrand \cite{Ta81} showed in ZFC (see also \cite{Pl02}) that if $K$ admits a measure of type $\ge \omega_2$
then $P(K)$ can be continuously mapped onto $[0,1]^{\omega_2}$. But then $P(K)$ cannot be $\azm$, as the property
is preserved by taking continuous images of compacta. Hence,  $(a)$ follows from Talagrand's result.

To check $(b)$,
let $S\sub K$ be the support of a measure $\mu$ of countable type. Then $P(S)$ can be seen as a subspace of $P(K)$; $P(S)$ is separable
by Lemma \ref{added}. Consequently, $P(S)$  is metrizable and  $S$ is also metrizable since $S$  embeds into $P(S)$.

Now to check $(c)$, it is enough to prove that under Martin's axiom $K$ cannot carry a measure of type $\omega_1$.
This follows from Theorem \ref{thm2}: otherwise, there is a continuous surjection $K\to [0,1]^\oo$; since $[0,1]^{\omega_1}$ is not $\azm$,
$K$ cannot be $\azm$ so neither can $P(K)$.
\end{proof}

\begin{corollary}\label{cor} Under Martin's axiom MA$(\omega_1)$,  the following are equivalent for a compact space $K$

\begin{enumerate}[(i)]
\item  $P(K)$ is $\aleph_0$-monolithic;
\item $K$ is $\aleph_0$-monolithic;
\item the support of every $\mu\in P(K)$ is metrizable.
\end{enumerate}
\end{corollary}

\begin{proof}
 $(i)\to (ii)$ and $(iii)\to (i)$ hold by Lemma \ref{lem}.

To verify $(ii)\to (iii)$ recall that the support of $\mu\in P(K)$ is $ccc$ and apply Theorem \ref{thm1}. Alternatively,
we can use Theorem \ref{thm2} again.
\end{proof}

\begin{remark}
It seemed natural to recall Fremlin's result in our context.
Let us remark, however,   that Proposition \ref{prop}$(c)$ could be  also derived from Theorem \ref{thm1} alone: if $S\sub K$ is the support of $\mu\in P(K)$ then
one can check that $P(S)$ is also $ccc$ subspace of $P(K)$ and it follows that  $P(S)$ is $\azm$.
%
%Theorem \ref{thm1} is in fact a direct consequence of Martin's axiom modulo two reductions: suppose that $K$ is a nonmetrizable $ccc$ space which is $\azm$.
%Then it can be mapped onto a space $K_1$ of weight $\oo$ which is still nonmetrizable ($K_1$ is necessarily $ccc$ and $\azm$).
%Using $ccc$ we can find a nonempty closed $ccc$ space $K_2\sub K_1$ such that no nonempty open subset of $K_2$ is metrizable.
%Then $K_2$ is the closure of some ${\{x_\alpha:\alpha<\oo\}\sub K_2$ and $F_\alpha=\overline{\{x_\beta:\beta<\alpha\}}$
\end{remark}

The implication $(ii)\to (i)$ of  Corollary \ref{cor} is not provable in the usual set theory. Kunen \cite{Ku81} constructed under CH a nonseparable compact space $K$ which is Corson compact (hence $\azm$) and such that $K$ supports a measure $\mu\in P(K)$ of countable type (see
\cite[ the remark on page 287]{Ku81}). Then $P(K)$ is separable but nonmetrizable so $P(K)$ is not $\azm$.
In fact, it can be derived from a result due to Talagrand \cite{Ta80} that under CH there is a Corson compact space $K$ such that
$P(K)$ contains a copy of $\beta\omega$, so the monolithicity of $P(K)$ is dramatically violated.
The status of $(i)\to (iii)$ of  Corollary \ref{cor} seemed to be unclear; to state this explicitly we arrive at the following question.

\begin{problem}\label{problem1}
Can one prove in ZFC that whenever $P(K)$ is $\aleph_0$-monolithic then the support of every $\mu\in P(K)$ is metrizable?
\end{problem}

Problem \ref{problem1} was communicated to us a couple of years ago by Wies{\l}aw Kubi\'s in connection with \cite{KK12} and \cite{FKK13}.
More recently, the same question was asked by Claudia Correa who noted that a  positive answer to \ref{problem1} would  provide a handy
characterization of those compact $K$ for which $P(K)$ is $\azm$ (see \cite{Co19}). We shall show, however,  that this is not the case.

\begin{theorem}\label{main}
Under $\diamondsuit$,  there is a nonmetrizable  Corson compact space $K$ such that $P(K)$ is
$\aleph_0$-monolithic but $K$ supports a measure of type $\omega_1$.
\end{theorem}

 The construction that is behind our main result  is a variant of Kunen's contruction from \cite{Ku81}
done in the spirit of \cite{Pl97}.
We should recall that Kunen's primary construction from \cite{Ku81} gave $K$ supporting a measure of uncountable type.
However, it seems that one needs to add a number of new ingredients to the inductive process to guarantee that $P(K)$ is indeed $\azm$.
Moreover, our construction requires $\diamondsuit$ and we do not know if \ref{main} follows from CH.
It is worth recalling that Kunen's construction was also used by Brandsma and van Mill \cite{BvM98}, to give
an example of a compact HL space with a non-monolithic hyperspace.

Recall finally that for a Corson compact space $K$, the space $P(K)$ is Corson compact if and only if the support of every $\mu\in P(K)$ is
metrizable, see \cite{AMN}. Hence, the space $P(K)$ announced in Theorem \ref{main} is $\azm$ but not Corson compact.

\section{Measures on some Boolean algebras}

In this section we discuss properties of finitely additive measures on Boolean algebras. If $\cG$ is a subset of a Boolean algebra $\fA$ then $[\cG]$ denotes
the smallest subalgebra of $\fA$ containing $\cG$.

Let us fix   a Boolean algebra $\fA$ for a while; we denote by $P(\fA)$ the space of all finitely additive probability measures on $\fA$.
If we consider $K=\ult(\fA)$,  the Stone space of $\fA$, then   we can speak of
$P(\fA)$ rather than of $P(K)$. Indeed,  every measure on $K$ is uniquely determined by its restriction to the algebra $\clop(K)$ of clopen subsets of $K$, which is
isomorphic to $\fA$. In the other direction, every $\mu\in P(\fA)$ uniquely defines the measure $\wh{\mu}\in P(K)$, where
$\wh{\mu}(\wh{a})=\mu(a)$ for $a\in \fA$. Here $a\to\wh{a}$ denotes the Stone isomorphism between $\fA$ and $\clop(K)$.
Then the $weak^\ast$ topology on $P(K)$ becomes the topology on $P(\fA)$ of convergence on elements of $\fA$.

\begin{lemma}\label{p:1}
Given an algebra $\fA$, the space $P(\fA)$ is monolithic if and only if for every countable set $E\sub P(\fA)$ there is a countable
subalgebra $\fA_0$ of $\fA$ such that every  sequence of $\mu_n\in E$ converging  on $\fA_0$  converges also on $\fA$.
\end{lemma}

\begin{proof}
It is easy to see that the condition is necessary. For the sufficiency,  note that if $\fA_0$ is such  a test subalgebra for a set $E$ then elements
of $\fA_0$ separate the set of measures $\overline{E}$.

Indeed, take $\mu,\nu\in \overline{E}$ such that $\mu$ and $\nu$ agree on $\fA_0$ and consider any  $b\in\fA$.
Since $\fA_0$ is countable, there are $\mu_n,\nu_n\in E$ such that $\mu_n\to\mu$ on $[ \fA_0\cup\{b\}]$
and $\nu_n\to\nu$ on $[ \fA_0\cup\{b\}]$. Then the sequence $\mu_1,\nu_1,\mu_2,\mu_2,\ldots$ converges to $\mu|\fA_0=\nu|\fA_0$
on $\fA_0$ so it converges on $\fA$; in particular,  $\mu(b)=\nu(b)$.
\end{proof}

Denote by $\lambda$ the usual product measure on the Cantor cube $2^{\omega_1}$ defined on the product $\sigma$-algebra $Ba({\omega_1})$
(of Baire subsets of $2^{\omega_1}$).
The algebra $Ba(\omega_1)$ is $\sigma$-generated by the algebra $\clop(\omega_1)$ of clopen sets.
It will be convenient to use the following notation.

\begin{notation}\label{not}
Given a subalgebra $\fA$ of $Ba(\omega_1)$ and  $I\sub\omega_1$,  we write $\fA(I)$ for the family of those $A\in \fA$ which are determined by coordinates in the set $I$.
\end{notation}

In particular, if $\alpha<\oo$ then $Ba(\alpha)$ is family of all Baire sets determined by coordinates in $\{\beta: \beta<\alpha\}$.
Accordingly,  $\clop(\alpha)$ is the family of closed-and-open subsets of $2^\oo$ that are determined by coordinates below $\alpha$.

We shall  frequently use the fact that $\lambda$ is a product measure: if $A\in Ba(I)$ and $B\in Ba(J)$, where $I\cap J=\emptyset$, then $\lambda(A\cap B)=\lambda(A)\cdot\lambda(B)$.

We collect below some preliminary facts concerning measures on $Ba(\omega_1)$ and its subalgebras.
Recall  that, given  two finitely additive measures $\mu$ and $\lambda$ defined on a Boolean algebra $\fA$, we say that
$\mu$ is absolutely continuous with respect to $\lambda$ ($\mu \ll\lambda$) if for every $\eps>0$ there is $\delta>0$ such that
for all $a\in\fA$, if $\lambda{(a)}<\delta$ then $\mu(a)<\eps$.

The first lemma just rephrases the classical Radon-Nikodym theorem.

\begin{lemma}\label{p:2}
Let $\fA$ be any  subalgebra of $Ba(\omega_1)$. If $\mu$ is a finitely additive finite measure on $\fA$ which is  is absolutely continuous with respect to $\lambda$ on $\fA$ then
there is a function $g:2^\oo\to\er$, which is measurable with respect to the $\sigma$-algebra generated by $\fA$, such that
$\mu(A)=\int_A g\; {\rm d}\lambda$ for every $A\in\fA$.
\end{lemma}

\begin{lemma}\label{p:2.5}
Suppose that
for $\beta<\alpha$ and two algebras $\fB,\fF\sub Ba(\omega_1)$ we have
   $\fB(\alpha)\sub \left[ \fB(\beta) \cup \fF(\oo\sm\alpha)\right]$.
Then every $B\in\fB(\alpha) $ is a finite disjoint union of sets of the form $B_1\cap B_2$, where $B_1\in\fB(\beta)$ and $B_2\in \fF(\oo\sm\alpha)$.
\end{lemma}

\begin{proof}
This is standard: it is enough to check that every set from the algebra $\left[ \fB(\beta) \cup \fF(\oo\sm\alpha) \right]$ has the required property.
\end{proof}

\begin{lemma}\label{p:3}
Let $\fB\sub \fF$ be two subalgebras of $Ba(\omega_1)$ such that for some  cofinal set $ S\sub\oo$
the following are satisfied

\begin{enumerate}[(i)]
\item $\fB(\alpha)$ is countable for every $\alpha\in S$;
\item whenever $\beta,\alpha\in S$  then $\fB(\alpha)\sub\left[ \fB(\beta)\cup \fF(\alpha\sm\beta)\right] $.
\end{enumerate}

Suppose  that,  for every $n$, $\mu_n\in P(\fF)$ is such a measure
that $\mu_n|\fF(\omega_1\sm\xi_n)$ is absolutely continuous with respect to $\lambda$ for some $\xi_n\in S$.
Then the closure of $\{\mu_n|\fB: n<\omega\}$ in $P(\fB)$ is metrizable.
\end{lemma}

\begin{proof} We consider first a single measure $\mu\in P(\fF)$ that is absolutely continuous with respect to $\lambda$
on $\fF(\omega_1\sm \xi)$ for some $\xi\in S$.
\medskip

\noindent {\sc Claim.} There is $\xi< \eta_0<\oo$ such that whenever $\eta \in S\sm \eta_0$,
$A\in \fB(\xi)$, $B\in \fF(\eta\sm\xi)$, $C\in\fF(\omega_1\sm\eta)$  then
$\mu(A\cap B\cap C)=\mu(A\cap B)\cdot\lambda(C)$.

Consequently, $\mu(G\cap C)=\mu(G)\cdot \lambda(C)$ for $G\in \fB(\eta)$, $C\in\fF(\omega_1\sm\eta)$.
\medskip

To check the claim  we take any  $A\in \fB(\xi)$ and
 apply Lemma \ref{p:2} to the measure
\[\fF(\omega_1\sm \xi)\ni B \to \mu(A\cap B).\]
By Lemma \ref{p:2} there is
an $Ba(\oo\sm\xi)$--measurable function $g_A:2^\oo\to\er$ such that $\mu(A\cap B)=\int_B g_A\; {\rm d}\lambda$ for
every $B\in \fF(\oo\sm\xi)$. Since every $Ba(\oo)$--measurable function is determined by countably many coordinates (and since $\fB(\xi)$ is countable),
there is $\eta_0\in S$ such that $\xi<\eta_0<\oo$
and $g_A$ is $Ba(\eta_0\sm\xi)$--measurable for every $A\in\fB(\xi)$.
If we now take any $\eta\ge\eta_0$, $A\in \fB(\xi)$, $B\in \fF(\eta\sm\xi)$, $C\in\fF(\omega_1\sm\eta)$
 then
\[ \mu(A\cap B\cap C)=\int_{B\cap C} g_A\; {\rm d}\lambda=\lambda(C)\cdot \mu(A\cap B),\]
by stochastic independence. This proves the first statement. The second one follows from Lemma \ref{p:2.5}
and $(ii)$ --- such a set $G$ in question is a finite union of sets of the form $A\cap B$ where $A\in \fB(\xi)$, $B\in \fF(\eta\sm\xi)$.
\medskip

Coming back to  a sequence of measures  $\mu_n$ as in the assumption, it follows from Claim that there is a single $\eta<\omega_1$
 such that  for every $n$  we have
 \[ (*) \quad \mu_n(G\cap C)=\mu_n(G)\cdot \lambda(C)
 \mbox{ whenever  } G\in \fB(\eta), C\in\fF(\omega_1\sm\eta).\]
To conclude the argument, in view of Lemma \ref{p:1} it suffices to check that any subsequence of $\mu_n$'s converges
on $\fB$ whenever it converges on the countable algebra $\fB(\eta)$. This follows from (*) and the fact  that
every $H\in \fB$ belongs to some $\fB(\alpha)$ for $\alpha$ large enough so $H$ is a finite union
of intersection $G\cap C$ as in (*).
\end{proof}

\section{Construction}

Write $\lo$ for the set of all limit ordinals in $\oo$.
Recall that Jensen's diamond principle  $\diamondsuit$ declares the existence of a sequence $\la S_\alpha:\alpha<\oo\ra$ with
$S_\alpha\sub\alpha$ such that the set $\{\alpha<\oo: X\cap\alpha=S_\alpha\}$ is stationary for every $X\sub\omega_1$
(\cite[\S 7]{Kunen} or \cite{Jech}, page 191).

 We shall use $\diamondsuit$ in the following form.

\begin{lemma} \label{diamond}
Under $\diamondsuit$, there is a sequence $\langle \nu_\alpha: \alpha\in \lo \rangle$ of finitely additive measures on $Ba(\oo)$
such that for every continuous increasing sequence $\la \fF_\alpha:\alpha<\oo\ra$ of countable subalgebras of $Ba(\oo)$ and for every
$\mu\in P(Ba(\omega_1))$ the set
\[ \{\alpha\in\lo: \mu|\fF_\alpha=\nu_\alpha|\fF_\alpha\},\]
is stationary.
\end{lemma}

\begin{proof}
Since $\diamondsuit$ implies CH and $|Ba(\oo)|=\con$ we can write $Ba(\oo)$ as a union $\bigcup_{\alpha<\oo} \fB_\alpha$ of a continuous increasing chain of some countable algebras.
By the standard coding using $\diamondsuit$, we can find $\langle \nu_\alpha: \alpha\in \lo \rangle$ such that for every $\mu\in P(Ba(\omega_1))$ the set
\[ S=\{\alpha\in\lo: \mu|\fB_\alpha=\nu_\alpha|\fB_\alpha\},\]
is stationary,  compare \cite[Exercise 51]{Kunen}.

Consider any continuous increasing chain of countable algebras $\la \fF_\alpha:\alpha<\oo\ra$ with the union $\fF$. It is easy to check that  the set
\[ T=\{\alpha\in\lo: \fF_\alpha=\fB_\alpha\cap \fF\},\]
is closed and unbounded in $\oo$, so $S\cap T$ is stationary, and we are done .
\end{proof}

Below we define an increasing chain $\la \fA_\alpha: {\alpha\in\lo}\ra$ of countable algebras
 $\fA_\alpha \sub Ba(\alpha)$ so that
\[ \fA=\bigcup_{\alpha\in\lo} \fA_\alpha,\]
 will be the Boolean algebra we are looking for.
 At each step $\alpha$ we choose a countable family $\cG_\alpha$ (of new generators) and define
  $\fA_\alpha$ to be the algebra  generated by
$\bigcup_{\beta<\alpha} \fA_\beta \cup \cG_\alpha$.
We also use the notation
\[ \fB_\alpha=\left[ \fA_\alpha\cup\clop(\alpha)\right],\]
for auxiliary  algebras (to which we shall apply Lemma \ref{p:3}).

To state inductive assumptions we need another piece of notation. Given any $G\sub 2^{\omega_1}$ and $\alpha<\omega_1$
we write
\[\cyl_\alpha(G)=\pi^{-1}_\alpha\left[ \pi_\alpha[G]\right],\]
where $\pi_\alpha:2^{\omega_1}\to 2^\alpha$ is the usual projection. Note that for $\beta<\alpha$, if $G= A\cap B$ where $A\in Ba(\alpha)$, $\emptyset\neq B\in Ba(\alpha\sm\beta)$ then $\cyl_\beta(G)=A$.

Here is the list of requirements (in the sequel,  $\alpha,\beta, \ldots \in\lo$ and $i,j,k,n$ are natural numbers):

\begin{enumerate}[\bf R(1)]
\item $\lambda(A)>0$ for every nonempty $A\in \fA_\alpha$;
\item every $\cG_\alpha$ is enumerated as $\cG_\alpha=\{G(\alpha,n): n<\omega\}$, where  $G(\alpha,n)\sub G(\alpha,n+1)$ for every $n$ and  $\lim_n \lambda(G(\alpha,n))=1$;
\item if $\beta<\alpha$ then for every $i$ there is $j$  such that $G(\alpha,i)\sub  G(\beta,j)$;
\item if $\beta<\alpha<\oo$
then for almost all $i$ the set  $G(\alpha, i)$ is  of the form $G(\alpha,i)=A\cap B$ with $A\in\fA_\beta$, $B\in Ba(\alpha\sm\beta)$;
\item if $\beta<\beta'< \alpha$ then    for every $i$
\[ \cyl_\beta(G(\beta',i))\sub \cyl_\beta(G(\alpha,j))\quad \mbox{for almost all } j;\]
\item if $\beta< \alpha$ then $\fB_\alpha\sub\left[ \fB_\beta\cup Ba(\alpha\sm \beta)\right]$.
\end{enumerate}

We now fix the guessing sequence $\langle \nu_\alpha: \alpha\in \lo \rangle$ of Lemma \ref{diamond}.
Our construction is modelled by those measures $\nu_\alpha$ considered on the continuous increasing chain of algebras
$\langle \fF_\alpha: \alpha\in \lo \rangle$ that we define along with the construction.
The role of $\fF_\xi$ (for $\xi\in\lo$) is to remember what  happened below $\xi$:
we assume that $\fF_\xi$ contains $\bigcup_{\alpha<\xi} \fB_\alpha$ and all the witnesses for R(6) below $\xi$  so that
\[ \fB_\alpha\sub\left[ \fB_\beta\cup \fF_\xi (\alpha\sm \beta)\right],\]
for $\beta<\alpha<\xi$. We also put $\fF=\bigcup_{\xi\in\lo} \fF_\xi$.

We start by examining our freedom at the limit step of the construction.

\begin{lemma}\label{observe}
Suppose that we are given $\fA_\beta, \cG_\beta$ for $\beta<\alpha$ satisfying (an appropriate portion of) { R(1) --- R(6)}.
Suppose also that  $\alpha=\sup_n\beta_n$ for an increasing sequence of $\beta_n\in\lo$.

Then there is  a function $\vf:\omega\to\omega$ such that the sets
\[ G^\vf(i)=\bigcap_{n\ge i} G(\beta_n, \vf(n)),\]
and corresponding algebras satisfy { R(1 ) --- R(6)} (when we set $G(\alpha,i)=G^\vf(i)$).
\end{lemma}

\begin{proof}

Using R(2) we can define $\vf_0:\omega\to\omega$ so that $\lambda(G(\beta_n, \vf_0(n))> 1-1/2^{n+2}$ for every $n$.
Then, by elementary calculations, we have $\lambda(G^\vf(i))> 1-1/2^{i+1}$ for every function $\vf\ge\vf_0$
so the sets $G^\vf(i)$ satisfy R(2). Note that, in particular,  $G^\vf(0)\neq\emptyset$.

Fix a bijection $g:\alpha\to\omega$; we inductively define a function $\vf:\omega\to\omega$ so that $\vf\ge\vf_0$ and the following are satisfied:

\begin{enumerate}[(a)]
\item if $g(\beta)< n,  \beta<\beta_n$ then

$G(\beta_n,\vf(n))=A_n\cap B_n$, where  $A_n\in\fA_\beta, B_n\in Ba(\beta_n\sm\beta)$;
\item  if $g(\beta)< n,  \beta<\beta_k< \beta_n$ then

$ \cyl_\beta(G(\beta_n, \vf(n))\supseteq  \cyl_\beta(G(\beta_k, \vf(k))$;
\item if $\beta<\beta',  g(\beta)< n,  g(\beta')< n, k< n$ then

$\cyl_\beta(G(\beta_n, \vf(n))\supseteq  \cyl_\beta(G(\beta', k)$.
\end{enumerate}

Note that such $\vf$ can be defined by inductive assumptions since (a), (b) and  (c) require fulfilling  only a finite number of  conditions
at each step. We first  check that the sets $G^\vf(i)$ satisfy  R(2) --- R(6).
Note that R(2) follows from $\vf\ge\vf_0$.

Given $\beta<\alpha$ and $i<\omega$,  $\beta<\beta_n<\alpha$ for some $n>i $; then by R(3)  there is $j$ such that
$G(\beta_n,\vf(n))\sub G(\beta,j)$; hence $G^\vf(i)\sub G(\beta_n,\vf(n))\sub G(\beta,j)$;
this shows that R(3) is preserved.

To check that the sets $G^\vf(i)$ satisfy R(4) fix $\beta<\alpha$. If we take any $i$ with
$\beta<\beta_i$ and $g(\beta)< i$ then (using the notation as in (a))
\[ G^\vf(i)=\bigcap_{n\ge i} G(\beta_n,\vf(n))=\bigcap_{n\ge i} A_n\cap B_n=A_i\cap \bigcap_{n\ge i} B_n,\]
since $A_n\supseteq A_i$ for $n\ge i$ by (b). The formula above shows that $G^\vf(i)$ has the required form
 for almost all $i$.

We check R(5) and R(6) in a similar manner: for instance, given $\beta<\beta' <\alpha$ and any $i$,
 $\cyl_\beta(G^\vf(j))\supseteq  \cyl_\beta(G(\beta',i))$ for almost all $j$ by (c).

To treat R(1) note first that we can additionally demand that the function $\vf$ satisfies
$\lambda(G^\vf(i+1)\sm G^\vf(i))>0$ for every $i$. Then R(1) follows easily from the following observation.

Suppose that  $\lambda$ is strictly positive on some $\fA_\beta$. Let $G=A\cap B$, where $A\in\cA_\beta$,
$B\in Ba (\alpha\sm\beta)$, $\lambda(A),\lambda(B)>0$. Then $\lambda$ is strictly positive
on the algebra generated by $\cA_\beta$ and $G$.
\end{proof}

\begin{remark}\label{ra}
Lemma \ref{observe} is stated in the form presenting the main idea of a diagonal argument. However,
what we really need to know at the limit step of the construction in \ref{construction}  is slightly more complicated.
Suppose that in the setting of \ref{observe} we are additionally given a sequence of Baire sets $C_n\in Ba(\beta_{n+1}\sm\beta_n)$
(with $\lambda(C_n)$ growing fast to 1).
Then the above proof shows, after minor changes, that the sets
\[ G^\vf(i)=\bigcap_{n\ge i} C_n\cap G(\beta_n, \vf(n)),\]
also satisfy the assertion of \ref{observe} for some $\vf$.
\end{remark}

\begin{construction}\label{construction}
Suppose that the construction has been done for $\beta\le\alpha'$ and consider the next limit ordinal $\alpha=\alpha'+\omega$.
We use this step simply to add new sets to $\fA_{\alpha'}$: Choose any strictly increasing sequence of
$C_n\in \clop({\alpha\sm \alpha'})$ such that $\lim_n \lambda(C_n)=1$.
Define $G(\alpha,n)=G(\alpha',n)\cap C_n$ for $n<\omega$ and set
\[ \cG_\alpha=\{G(\alpha,n): n<\omega\}.\]
Checking that R(1) --- R(6) are preserved is fairly standard;
for future reference note the following.
%\medskip

\begin{remark}\label{ar}
There is $B\in\fA_\alpha$ such that
$\inf\{\lambda(A\bigtriangleup B): A\in\fA_{\alpha'}\}>0$.
Indeed, this holds whenever we take  $B=G(\alpha,n)=G(\alpha',n)\cap C_n$
with  $\lambda(C_n)>0$ since $C_n$ is independent from all elements in $\fA'$.
\end{remark}

Suppose now that the construction has been done below $\alpha$ which is a limit ordinal in $\lo$.

Let us say that {\bf Limit Case (1)}  happens if the following holds:  there are $\eps>0$ and $\beta_n, \alpha_n\in\lo$, where
\[ \beta_0 <\alpha_0\le \beta_1<\alpha_2\le\ldots<\alpha,\]
and there are $C_n\in \fF_{\alpha_n}(\alpha_n\sm \beta_n)$ such that
$ \lim_n \lambda(C_n)=1$ while  $\nu_\alpha(C_n)\le 1-\eps$ for every $n$.

When  { Limit Case (1)}  happens, we proceed as follows.

Note first that, passing to a subsequence if necessary, we can assume that $\lambda(C_n)>1-1/2^{n+2}$ for every $n$.
Then we put
\[ G(\alpha, i)=\bigcap_{n\ge i} C_n\cap G\left(\beta_n,\vf(n)\right),\]
where $\vf$ is chosen as in Lemma \ref{observe} which guarantee that the new generators satisfy R(1) --- R(6).
Note that the appearance of $C_n$'s in the formula above (those sets were not mentioned in \ref{observe}) does
not change much since every $C_n$ is determined by coordinates in $\alpha_n\sm\beta_n$, see Remark \ref{ra}.

We say that {\bf Limit Case (2)} happens in the remaining case.
We then perform the previous construction in a simpler form, taking above $C_n=2^\oo$ for every $n$.
\end{construction}

\section{Analyzing the resulting Stone space}

We shall now prove Theorem \ref{main} by analyzing the Stone space $K=\ult(\fA)$ of the Boolean algebra $\fA$ constructed in the previous section.

\begin{lemma}\label{l0}
For every $\alpha\in\lo$ the set
\[M_\alpha=K\sm \bigcup_n \widehat{G(\alpha,n)},\]
is a metrizable subspace of $K$
\end{lemma}

\begin{proof}
Indeed, if $\xi>\alpha$ then for every $k$ we have $G(\xi,k)\cap M_\alpha=\emptyset$ by R(3); it follows that
$\{\wh{A}\cap M_\alpha: A\in\fA_\alpha\}$ is a countable base of $M_\alpha$.
\end{proof}

\begin{lemma}\label{l1}
The space $K$ is a nonseparable  Corson compact space supporting a strictly positive measure of type $\omega_1$.
\end{lemma}

\begin{proof}
Since $\lambda(A)>0$ for every nonempty $A\in\fA$; the measure $\wh{\lambda}$, uniquely determined by the formula
$\wh{\lambda}(\wh{A})=\lambda(A)$  for $A\in\fA$, is strictly positive on $K$.

To see  that $\wh{\lambda}$ is a measure of uncountable type, note that by Remark \ref{ar},
for every $\alpha$ there is $B\in\fA$ such that
$\inf\{\lambda(A\bigtriangleup B): A\in\fA_\alpha\}>0$ so no countable subfamily of
$\fA$ can be $\btu$-dense in $\fA$ with respect to $\lambda$.
In particular, $K$ is not metrizable, as it carries a measure of uncountable type.

In order to check that   $K=\ult(\fA)$ is Corson compact it suffices to find a family
$\cG\sub\fA$ such that  $\fA=[\cG]$, and   having the property that
every centered $\cG_0\sub\cG$ is countable. Indeed, in such a case we have an embedding
\[ \Phi: \ult(\fA)\ni x\to \la \chi_{\wh{G}}(x): G\in\cG\ra,\]
into $\Sigma(2^{\cG})$. Here $\chi_{\wh{G}}$
denotes a characteristic function of the clopen set $\wh{G}\sub K$ so $\Phi$ is clearly continuous; the injectivity of $\Phi$
follows from the fact that $\cG$ generates $\fA$.

In our case we take $\cG=\{G(\alpha,n): \alpha\in\lo, n<\omega\}$. If $\cG_0\sub\cG$ is centered then there is
a 0-1 measure $\mu$ on $Ba(\oo)$ such that $\mu(G)=1$ for $G\in \cG_0$.
Then $\mu|\fF_\alpha$ was guessed at some limit step $\alpha$ by $\nu_\alpha$. Then, necessarily, Limit case (1) happened (recall
that, in particular, $\fF_\alpha$ contains $\clop(\alpha)$) so $\nu_\alpha(G(\alpha,n))<1$  and thus  $\nu_\alpha(G(\alpha,n))=0$
for every $n$.  By R(3), $G(\beta,n)\notin \cG_0$ whenever $\beta\ge\alpha$ and $n\in\omega$ so $\cG_0$ is indeed countable.

Finally, $K$ is nonseparable since every separable Corson compactum is metrizable.
\end{proof}

\begin{lemma}\label{l2}
Let $\mu\in P(Ba(\omega_1))$ be such a measure that $\wh{\mu|\fA}$ defines a regular Borel measure in $P(K)$ vanishing
on all closed metrizable subsets of $K$.

Then $\mu|\fF(\omega_1\sm\alpha_0)$ is absolutely continuous with respect to $\lambda$ for some $\alpha_0<\omega_1$.
\end{lemma}

\begin{proof}
Recall that the measures $\nu_\alpha$ often guess all the other  measures on algebras  $\fF_\alpha$.
Hence, by Lemma \ref{diamond} the set $S$ of those $\alpha\in \lo$ for which
${\mu}$ agrees with $\nu_\alpha$ on $\fF_\alpha$ is stationary.
\medskip

\noindent {\sc Claim.} Limit Case (1) happened for no $\alpha\in S$.
\medskip

Suppose otherwise, that (1) occurred for some $\alpha\in S$ and $\eps>0$. Then, using the notation of the construction,
${\mu}(G(\alpha,n))\le 1-\eps$ for every $n$. In other words, if we take
\[M_\alpha=K\sm \bigcup_n \widehat{G(\alpha,n)},\]
%M=\bigcap_n \widehat{G(\alpha,n)^c}\sub K,\]
then $\wh{\mu}(M_\alpha)\ge\eps$. To arrive at contradiction, it is sufficient to note that $M_\alpha$ is metrizable by Lemma \ref{l0}.
\medskip

Fix some $\eps>0$. We know from Claim that
\[  %\hspace{-2em}
\left(\forall \alpha\in S\right) \left(\exists \xi^\eps(\alpha)<\alpha\right) \left(\forall \xi^\eps(\alpha)<\beta<\alpha\right)
(\exists n) (\forall A\in \fF_\beta(\beta\sm  \xi^\eps(\alpha))\]
\[ \lambda(A)<1/n\Rightarrow {\mu}(A)<\eps.\]
By the pressing down lemma, there is $\alpha_0^\eps<\omega_1$ such that $\xi^\eps(\alpha)\le \alpha_0^\eps$ for stationary many $\alpha\in S$.
Repeating  this argument for every $\eps=1/k$, we conclude that there is $\alpha_0<\omega_1$ such that, for all $\eps>0$,
$\xi^\eps(\alpha)\le\alpha_0$ for $\alpha$ from some stationary set $T_\eps \sub S$.

It follows that ${\mu}$ is absolutely continuous with respect to $\lambda$ on $\fF_\beta({\beta\sm \alpha_0})$ for every $\beta>\alpha_0$.
Hence,
${\mu}\ll\lambda$ on $\fF({\omega_1 \sm \alpha_0})$, as required.
\end{proof}

We are finally ready to verify the main point.

\begin{theorem}
The space $P(K)$ is monolithic.
\end{theorem}

\begin{proof}
Given  any  measure $\mu$ in $P(K)$,
consider
\[c=\sup\{\mu(L): L\sub K, L\mbox{ is closed and metrizable}\}.\]
 Take a sequence of closed metrizable subspaces $L_n$ with
$\mu(L_n)>c-1/n$. Then $L=\overline{\bigcup_n L_n}$ is again metrizable because,  by Lemma \ref{l1}, $K$ is Corson compact hence monolithic. We have $\mu(L)=c$ so
$\mu(M)=0$ for every closed metrizable $M\sub K\sm L$.

It follows that every measure from $P(K)$ is a convex combination of a measure concentrated on a metrizable subspace of $K$ and a measure
vanishing on all closed metrizable subspaces of $K$.
Therefore, it is sufficient to check that the closure of $\{\mu_n: n<\omega\}\sub P(K)$
in $P(K)$ is metrizable whenever every $\mu_n$ vanishes on metrizable subsets of $K$.
This is a direct consequence of Lemma \ref{l2} and Lemma \ref{p:3}.
\end{proof}

\subsection*{Acknowledgements}
I wish to thank Witold Marciszewski for our discussion on $\azm$ spaces. I am indebted to the referee for his/her careful reading
and several comments that enabled me to clarify  the main idea and improve the presentation.

The research has been supported by the grant 2018/29/B/ST1/00223 from National Science Centre, Poland.

%%%%%%%%%%% To ease editing, use normal size for the references:

\end{document}